\documentclass[12pt]{article}
\usepackage{a4wide}
\usepackage{theorem}
\usepackage{amsfonts}
\usepackage{amssymb}
\usepackage{enumerate}
\usepackage{amsmath}
\usepackage{times}
\usepackage{amscd}
\usepackage[all]{xy}
\usepackage{verbatim}
\usepackage{graphicx}

\theoremstyle{change}
\theorembodyfont{\sl}
\newtheorem{thm}[subsection]{Theorem.}

\newtheorem{lem}[subsection]{Lemma.}

\newtheorem{conj}[subsection]{Conjecture.}

\theorembodyfont{\rmfamily}

\makeatletter
\newenvironment{eqn}{\refstepcounter{subsection}
$$}{\leqno{\rm(\thesubsection)}$$\global\@ignoretrue}
\makeatother  







\newenvironment{prf}[1]{\trivlist
\item[\hskip \labelsep{\it
#1.\hspace*{.3em}}]}{~\hspace{\fill}~$\square$\endtrivlist}

\newenvironment{proof}{\begin{prf}{\bf Proof}}{\end{prf}}

\newcommand{\ol}{\overline}

\newcommand{\zz}{{\mathbb Z}}

\newcommand{\qq}{{\mathbb Q}}

\newcommand{\cc}{{\mathbb C}}
\newcommand{\ff}{{\mathbb F}}

\newcommand{\pp}{{\mathbb P}}

\newcommand{\FF}{{\mathcal F}}

\newcommand{\Qbar}{{\overline{\qq}}}

\newcommand{\rH}{{\rm H}}

\newcommand{\rR}{{\rm R}}

\newcommand{\Sing}{\mathrm{Sing}}
\newcommand{\GL}{{\rm GL}}

\newcommand{\trace}{{\rm trace}}

\newcommand{\tors}{{\rm tors}}

\newcommand{\rank}{{\rm rank}}

\newcommand{\wt}{\widetilde}

\newcommand{\CH}{\mathrm{CH}}
\newcommand{\sing}{\mathrm{sing}}
\newcommand{\ch}{\mathrm{ch}}
\newcommand{\td}{\mathrm{td}}
\newcommand{\hor}{\mathrm{hor}}

\begin{document}

\author{Bas Edixhoven, Robin de Jong, Jan Schepers}
\title{Covers of surfaces with fixed branch locus}
\date{\today}

\maketitle

\begin{abstract} 
Given a connected smooth projective surface $X$ over~$\cc$, together
with a simple normal crossings divisor $D$ on it, we study finite
normal covers $Y \to X$ that are unramified outside~$D$. Given
moreover a fibration of $X$ onto a curve $C$, we prove that the
`height' of $Y$ over $C$ is bounded linearly in terms of the degree of
$Y \to X$. We indicate how an arithmetic analogue of this result, if
true, can be auxiliary in proving the existence of a polynomial time
algorithm that computes the mod-$\ell$ Galois representations
associated to a given smooth projective geometrically connected
surface over~$\qq$. A precise conjecture is formulated.
\end{abstract}

\section{Introduction}\label{sec1}

In this paper, we suppose given the following data:
\begin{itemize}
\item a connected smooth projective surface $X$ over~$\cc$;
\item a simple normal crossings divisor $D$ on~$X$ 
(i.e., all components of $D$ are smooth, and they intersect transversally);
\item a connected smooth projective curve $C$ over $\cc$;
\item a flat morphism $h \colon X \to C$.
\end{itemize}
We emphasise that we do not require the fibres of~$h$ to be connected.
We denote by $U$ the complement of $D$ in $X$. We are interested in
not necessarily connected finite \'etale covers $V \to U$; these are
considered to be the `variable' in our set-up. Given a finite \'etale
cover $V \to U$ denote by $\pi \colon Y \to X$ the normalisation of
$X$ in the product of the function fields of the connected components
of~$V$. By `finiteness of integral closure', the map $\pi$ is
finite. As the topological fundamental group of $U$ is finitely
generated (cf. \cite[Expos\'e II, Th\'eor\`eme 2.3.1]{sga7}), there
are only finitely many $V \to U$ of a given degree. In particular, for
fixed degree, the height over $C$ of the associated covers $Y \to X$
is bounded. Our aim is to prove an effective version of this
result. Let $\rho \colon Y' \to Y$ be a minimal resolution of
singularities of~$Y$, and denote by $f \colon Y' \to C$ the composed
morphism $h \pi \rho$. Note that $Y'$ and $Y$ are projective and flat
over~$C$ ($C$ being a Dedekind scheme).
\begin{thm} \label{main} Let $h\colon X\to C$ and $U\subset X$ be given
as above. Then there is an integer $c$ such that for all finite
\'etale $\pi\colon V\to U$ we have, in the notation as above:
\[
|\deg \det \rR^\cdot f_* O_{Y'}|\leq c{\cdot}\deg(\pi)\, .
\]
\end{thm}
Here $\det \rR^\cdot f_* O_{Y'}$ stands for the determinant of
cohomology of~$O_{Y'}$, cf.~\cite{de}. This is an invertible sheaf on
$C$ with $c_1(\det \rR^\cdot f_* O_{Y'}) = c_1(\rR^0 f_* O_{Y'}) -
c_1(\rR^1 f_* O_{Y'})$ in the Chow ring of~$C$. According to
\cite[Theorem~3.6(v)]{vi2} the degree $\deg \det \rR^\cdot f_* O_{Y'}$
is non-negative if the fibres of $f$ are connected and the arithmetic
genus of the fibres of $f$ is positive.

Our proof uses the Grothendieck-Riemann-Roch theorem, intersection
theory on the normal surface $Y$ and precise information about the
minimal resolution of singularities of~$Y$. In the last section we
state an arithmetic analogue of our result as a conjecture, motivated
by a possible application to the complexity of counting points on
reductions over finite fields of a fixed surface over~$\qq$.

In a previous version of this text, our bound in Theorem~\ref{main}
was quadratic in~$\deg(\pi)$. H{\'e}l{\`e}ne Esnault and Eckart Viehweg showed
us how to deduce a linear upper bound, with a precise constant, from
Arakelov's inequality. This result is now included as
Theorem~\ref{thm_esnault-viehweg}. A closer inspection of our proof of
Theorem~\ref{main} also gave a linear bound.

A good reason to include several proofs of the main result of this
article is that this may help to find a proof that can be made to work
in the arithmetic context. Since more than 25 years now, it has been
tried to prove an arithmetic analogue of Arakelov's inequality, without
success so far. Such an analogue would have led to an effective version
of Faltings's theorem (previously Mordell's conjecture). The same
observations apply to the Bogomolov-Miyaoka-Yau inequality. On the other
hand, we do believe that our conjectural analogue
(Conjecture~\ref{conj}) of Theorem~\ref{main} is not too hard to
prove.

\subsection*{Acknowledgements}
This text is the result of a research seminar that took place in
Leiden in 2006--2007. Apart from the authors the participants were
Johan Bosman, Peter Bruin, Jacob Murre, Arjen Stolk and Lenny Taelman.
The authors thank these participants for their input. Most
participants were supported by VICI grant 639.033.402 from NWO
(Nederlandse Organisatie voor Wetenschappelijk Onderzoek), and
moreover, Robin de Jong by NWO's VENI grant 639.031.619. We also thank
H{\'e}l{\`e}ne Esnault and Eckart Viehweg for their contribution.

\section{Preliminaries}\label{sec2}
Let $\pi\colon V\to U$ be as in the statement of Theorem~\ref{main},
and let $V=\coprod_i V_i$ be the decomposition of $V$ into connected
components. Then also $Y=\coprod_i Y_i$, and $Y'=\coprod_i Y'_i$.  Let
$f_i:=f|_{Y'_i}$. It follows that $\deg\det\rR^\cdot f_*
O_{Y'}=\sum_i\deg\det\rR^\cdot f_{i,*} O_{Y'_i}$. Hence, in order to
prove Theorem~\ref{main}, we can and do assume that $V$ is connected.

Denote by $d$ the degree of $V \to U$. According to \cite[Lemma~2]{vi}
the map $\pi \colon Y \to X$ is finite locally free of rank~$d$.  We
write $D^\sing$ for the singular locus of $D$.
\begin{lem} \label{prelfirstpartone}
The singularities of\/ $Y$ occur in the inverse image under $\pi$ of
$D^\sing$. Furthermore, the map $\pi^{-1}(D - D^\sing) \to D - D^\sing$ is 
\'etale. 
\end{lem}
\begin{proof} We base our argument on a consideration of fundamental groups, as
in \cite[pp. 102--103]{bpv}. Let
$x$ be a closed point of $X$ lying on $D$ but not on $
D^\sing$.  Locally for the analytic topology we identify a
neighbourhood $W$ of $x$ in $X$ with the bi-disk $Z = \{ (z_1,z_2) \in
\cc^2 \colon |z_1|<1, |z_2|<1 \}$, identifying $x$ with the origin and
$D$ locally with the zero set of $z_1$. Let $y$ be a point of $Y$
mapping to $x$ and consider the connected component $B$ of $\pi^{-1}W$
that contains $y$.  We have then that $B - \pi^{-1}(D) \to W - D$ is a
connected finite degree topological covering. Thus $\Gamma =
\pi_*(\pi_1( B - \pi^{-1}(D))) $ is a subgroup of finite index of
$\pi_1(W - D)$. The latter is infinite cyclic; let $e$ be the index of
the subgroup $\Gamma$. As the map $W \to W$ given by $(z_1,z_2)
\mapsto (z_1^e,z_2)$ is a connected cover of $W$, homeomorphic above
$W-D$ to the covering $B - \pi^{-1}(D) \to W - D$, we have by a theorem of
Grauert-Remmert \cite[Expos\'e XII, Th\'eor\`eme 5.4]{sga1} that $B$ itself is analytically
isomorphic to $W$ and the holomorphic map $B \to W$ equivalent to the given map
$W \to W$. We deduce that $Y$ is regular above $D-D^\sing$ and that $\pi^{-1}(D -
D^\sing) \to D - D^\sing$ is \'etale.
\end{proof}
Write $D=\sum_{i \in I} D_i$ for the decomposition of $D$ into prime
components, and write $\pi^{-1}(D_i) = \sum_{j \in J_i} D_{ij}$ for
the decomposition into prime components of the inverse image with
reduced structure under $\pi$ of a $D_i$. For $i \in I$ and $j \in
J_i$ denote by $e_{ij}$ the ramification index of $\pi$ along $D_{ij}$
(i.e., the ramification index of $\pi$ at the generic point of
$D_{ij}$) and denote by $f_{ij}$ the degree of $D_{ij}$ over
$D_i$. For each $i \in I$ we have $\sum_{j \in J_i} e_{ij}f_{ij} =
d$. If $x$ is a closed point on $X$ and $y$ is a point of $Y$ mapping
to $x$ we denote by $d_y$ the rank of the completed local ring
$\widehat{O}_{Y,y}$ as a free module over $\widehat{O}_{X,x}$. For all
closed points $x$ on $X$ we have $\sum_{y \colon y \mapsto x} d_y =
d$. 

A point on a complex surface is said to have type $A_{n,q}$ if the
complete local ring at that point is isomorphic as a $\cc$-algebra to
the complete local ring at the image of $(0,0)$ of the quotient of
$\cc^2$ under the action $(w_1,w_2) \mapsto (\zeta_n w_1, \zeta_n^q
w_2)$, where $n$ and $q$ are integers with $n>0$, $\gcd(n,q)=1$ and where
$\zeta_n=\exp(2\pi i/n)$. For $n>1$ this type is known as the cyclic
quotient singularity of type $A_{n,q}$. For $n=1$ such a point is
nonsingular; this case is included for notational convenience. 
\begin{lem} \label{prelfirstparttwo}  
Let $y$ be a point of\/ $Y$ mapping to $D^\sing$, say $\pi(y) =x \in D_i \cap
D_{i'}$. \\
(i) There are unique 
$j \in J_i$ and $j' \in J_{i'}$ such that $y \in D_{ij} \cap D_{i'j'}$.\\
(ii) Let $e_{ij}$ be the ramification index of $\pi$ along $D_{ij}$,
and let $e_{i'j'}$ be the ramification index of $\pi$ along
$D_{i'j'}$. Then there are positive integers $n,m_1,m_2$ depending on
$y$ such that the following holds: $e_{ij}=nm_1$, $e_{i'j'}=nm_2$, the
rank $d_y$ of $\widehat{O}_{Y,y}$ over $\widehat{O}_{X,x}$ equals
$nm_1m_2$.  If $y$ is a singular point of\/ $Y$ then $y$ is a cyclic
quotient singularity of type $A_{n,q}$ for some positive integer $q$
with $\gcd(n,q)=1$.
\end{lem}
\begin{proof} As in the previous lemma we base our argument on a consideration
of fundamental groups, following \cite[pp. 102--103]{bpv}. We identify
a local neighbourhood $W$ of $x$ with the bi-disk $Z= \{ (z_1,z_2) \in
\cc^2 \colon |z_1|<1, |z_2|<1 \}$, letting $D_i$ correspond to the
zero set of $z_1$ and $D_{i'}$ to the zero set of $z_2$. Let $Z^*= Z-
\{z_1z_2=0\}$ and $W^*$ the corresponding open subset of $W$. If
$\gamma_i \in \pi_1(Z^*)$ is the class of a positively oriented little
loop around the $z_i$-axis then $\pi_1(Z^*) \cong \zz \times \zz$ with
generators $\gamma_1=(1,0)$ and $\gamma_2=(0,1)$. Let $B$ be the
connected component of $\pi^{-1}W$ that contains $y$. Put $B^* = B -
\pi^{-1}D$. We have then that $B^* \to W^*$ is a connected finite
degree topological covering. Let $\Gamma=\pi_*(\pi_1(B^*))$ be the
image of the topological fundamental group of $B^*$ in $\pi_1(W^*)
\cong \zz \times \zz$. Then $\Gamma$ is of finite index in
$\pi_1(W^*)$. We pick generators of $\Gamma$ as follows: $\Gamma \cap
(\zz \times 0)$ is non-trivial, so there is a unique $n'>0$ such that
$(n',0)$ generates this intersection. As the quotient $\Gamma/\zz
(n',0)$ is isomorphic to $\zz$ there is a unique $(q',m_2) \in \Gamma$
with $0\leq q' < n'$ and $m_2>0$ such that $(n',0)$ and $(q',m_2)$
generate $\Gamma$.  Let $m_1 = \gcd(n',q')$ and write $n'=nm_1$,
$q'=qm_1$. Thus $\Gamma = \zz(nm_1,0)+\zz(qm_1,m_2)$ which is of index
$nm_1m_2$ in $\zz \times \zz$. Let $\Gamma'=\zz(nm_1,0)+\zz(0,nm_2)
\subset \Gamma$. Then $\Gamma'$ is the largest subgroup of $\Gamma$ of
the form $\zz(*,0)+\zz(0,*)$, and the quotient $\Gamma/\Gamma'\subset
m_1\zz/nm_1\zz\times m_2\zz/nm_2\zz$ is cyclic of order~$n$ and
projects isomorphically to both factors. Let $\widetilde{W}^*\to W^*$
be a universal cover. Then $B^*=\widetilde{W}^*/\Gamma$. Now
$\widetilde{W}^*/\Gamma'\to W^*$ is isomorphic to the cover $Z^*\to
Z^*$, given by $(z_1,z_2)\mapsto (z_1^{nm_1},z_2^{nm_2})$. Via
normalisation this induces the ramified cover $Z\to Z$, given by the
same formula. The group $\Gamma/\Gamma'$ acts on $Z$, with
quotient~$B$; this action is free outside~$(0,0)$. Hence $y$ is a
singular point of $Y$ if and only if $n>1$. We conclude that
$e_{ij}=nm_1$ and $e_{i'j'}=nm_2$ and we have natural inclusions of
$\cc$-algebras $\cc[[z_1,z_2]] \cong \widehat{O}_{X,x} \rightarrowtail
\widehat{O}_{Y,y} = \cc[[z_1^{1/nm_1},
z_2^{1/nm_2}]]^{\Gamma/\Gamma'}$, where the generator given by
$(qm_1,m_2)$ acts as $z_1^{1/nm_1}\mapsto \zeta_n^qz_1^{1/nm_1}$,
$z_2^{1/nm_2}\mapsto \zeta_n z_2^{1/nm_2}$. This realises
$\widehat{O}_{Y,y}$ as a direct summand of the free
$\widehat{O}_{X,x}$-module
$\cc[[z_1^{1/nm_1},z_2^{1/nm_2}]]$. Statement (i) follows. We see that
$\widehat{O}_{Y,y}$ is free of rank $nm_1m_2$ as
$\widehat{O}_{X,x}$-module whereas $\cc[[z_1^{1/nm_1}, z_2^{1/nm_2}]]$
is finite free of rank $n^2m_1m_2$ as $\widehat{O}_{X,x}$-module. We
get $d_y = nm_1m_2$, and if $n>1$, the singularity $y$ is a cyclic
quotient singularity of type $A_{n,q}$. The lemma is proved.
\end{proof} 
Let $\rho \colon Y' \to Y$ be a minimal resolution of singularities of
$Y$ and denote by $E_1,\ldots,E_s$ the exceptional components of $Y'
\to Y$. Let $K_Y$ be the Weil divisor obtained by taking the closure
in $Y$ of a canonical divisor on the non-singular locus of~$Y$. Since
$Y$ has only cyclic quotient singularities each Weil divisor on $Y$ is
$\qq$-Cartier, i.e., has the property that a certain integer multiple
of it is a Cartier divisor on $Y$. Let $\rho^*K_Y$ be the pull-back of
the $\qq$-Cartier divisor~$K_Y$. On $Y'-\cup_iE_i$, $\rho^*K_Y$ is a
canonical divisor. Hence for a canonical divisor $K_{Y'}$ of $Y'$
there is a linear equivalence of $\qq$-Cartier divisors:
\[ 
K_{Y'} \equiv \rho^* K_Y + \sum_{i=1}^s a_i E_i,
\]
where the $a_i$ are rational numbers. Such $a_i$ are unique, because
any rational function on $Y'$ whose divisor is a linear combination of
the $E_i$ is a rational function on $Y$ that is regular outside the
singular locus, and hence constant.

We have the following local statement. Let $y$ in~$Y$ be a singular
point, hence a cyclic quotient singularity, say of
type~$A_{n,q}$. Then for any Weil divisor $W$ on~$Y$, $n{\cdot}W$ is
Cartier at~$y$ (see \cite[Prop.~5.15]{kollar-mori} and its proof). The
minimal resolution of $y$ is described in \cite[pp. 100-101]{bpv}. We
have $n>1$, and we can and do assume that $q<n$. Write:
\[
\frac{n}{q} = b_1 - \frac{1}{b_2 - \frac{1}{b_3-\cdots}} =
[b_1,\ldots,b_\lambda]  
\]
with $b_i \in \zz_{>0}$ for the Hirzebruch-Jung continued fraction of
$n/q$. Then the reduced exceptional locus $\rho^{-1}(y)$ is a chain of
$\lambda$ $\pp^1$'s with self-intersections $-b_1,-b_2, \ldots,
-b_\lambda$. We will need some estimates related to this resolution of
singularities.
\begin{lem} \label{prelsecond} Let $y$ be a singular point of\/~$Y$ 
of type $A_{n,q}$. Assume that $E_1,\ldots,E_\lambda$ are the $E_i$
above~$y$, numbered as they appear in the chain. For $i$ in
$\{1,\ldots,\lambda\}$, let $b_i=-(E_i,E_i)$. Then $2 \leq b_i \leq n$
for each $i=1,\ldots,\lambda$. The number $\lambda$ of components is
bounded by $n$ as well. The $a_i$ are determined by the recursion
$b_ia_i - a_{i-1}-a_{i+1}=2-b_i$ with boundary conditions
$a_0=0=a_{\lambda+1}$. We have $a_i \in (-1,0]$ for
$i=1,\ldots,\lambda$ and the rational number $\sum_{i=1}^\lambda
a_i(b_i-2)$ is bounded from above by~$2$ and from below by $-n$.
\end{lem} 
\begin{proof} That
$b_i \geq 2$ is clear from the way the $b_i$ are defined. 
Spelling out the definition, 
the integers $b_i$ are determined by the following recursion
(variant of the Euclidean algorithm) for integers~$c_i$: $c_i =
b_{i+2}c_{i+1} - c_{i+2}$, $0 \leq c_{i+2} < c_{i+1}$ for
$i=-1,\ldots,\lambda -2$ with initial conditions $c_{-1}:=n,c_0:=q$. It
follows that $b_i \leq c_{i-2} \leq c_{-1}=n$. Further we have 
$n=c_{-1}>c_0>c_1>\ldots>c_\lambda=0$ so that
the number $\lambda$ is bounded from above by $n$. 

By the adjunction formula we have $(K_{Y'} + E_i,E_i) = -2$ for
all~$i$ so we see that the $a_i$ form the unique solution to the
recursion $b_ia_i - a_{i-1}-a_{i+1}=2-b_i$ for $i=1,\ldots,\lambda$
with boundary conditions $a_0=0=a_{\lambda+1}$.

Suppose that there is an $i$ in $\{1,\ldots,\lambda\}$ with
$a_i\geq0$. Let $j$ with $1 \leq j \leq \lambda$ be an index with $a_j
= \max_i a_i$. We find:
\[ 
a_j = \frac{a_{j-1} + a_{j+1}}{b_j} + \frac{2-b_j}{b_j} = \frac{2a_j}{b_j} +
\frac{a_{j-1}-a_j+a_{j+1}-a_j}{b_j} + \frac{2-b_j}{b_j} \leq a_j +
\frac{a_{j-1}-a_j}{b_j} + \frac{a_{j+1}-a_j}{b_j} 
\]
whence $a_{j-1}=a_j=a_{j+1}$ and $b_j=2$.  Hence the maximum of the
$a_i$ is also attained at $j-1$ and $j+1$. Continuing with the same
reasoning we find that all $b_i=2$ and all $a_i=0$. Hence all $a_i
\leq 0$. 

Let $j$ with $1 \leq j \leq \lambda$ be an index with $a_j =
\min_i a_i$. Our recursion can be written as
$(b_i-2)(a_i+1)=(a_{i-1}-a_i) + (a_{i+1}-a_i)$, so we see that if
$a_j\leq -1$, then $a_{j-1}=a_j=a_{j+1}$ and $a_{j-1}$ and $a_{j+1}$
are also minimal and~$\leq -1$, and we get the contradiction
$0=a_0\leq -1$. Hence for all~$i$ we have $a_i > -1$.

By adding the equalities $c_i =
b_{i+2}c_{i+1} - c_{i+2}$ for $i=-1,\ldots,\lambda-2$ we find:
\[ 
n+q+1 + 2 \sum_{i=1}^{\lambda-2} c_i  = \sum_{i=1}^\lambda b_i c_{i-1} 
\] 
and hence:
\[
n+q+1 = \sum_{i=1}^\lambda (b_i-2)c_{i-1} + 2c_{\lambda-1} + 2c_0 =
\sum_{i=1}^\lambda (b_i-2)c_{i-1} +2 + 2q \, . 
\] 
Since $c_{i-1} \geq 1$ and $b_i \geq 2$ for $i=1, \ldots, \lambda$ we
find $\sum_{i=1}^\lambda (b_i-2) \leq n-q-1 < n$. Adding the
equalities $b_ia_i - a_{i-1}-a_{i+1}=2-b_i$ for $i=1,\ldots,\lambda$
we find that $\sum_{i=1}^\lambda a_i(b_i-2) = \sum_{i=1}^\lambda
(-b_i+2) -(a_1 + a_\lambda)$. Since $\sum_{i=1}^\lambda (-b_i+2) > -n$
and $a_1 + a_\lambda \leq 0$ we find that $\sum_{i=1}^\lambda
a_i(b_i-2) > -n$. The upper bound $\sum_{i=1}^\lambda a_i(b_i-2) \leq
2$ follows since $b_i \geq 2$ always and $a_1+a_\lambda > -2$.
\end{proof}
We will need to compare the topological Euler characteristics of $X$
and~$Y$.  The following general lemma is useful for this. We denote by
$\rH^\cdot_c(-,\qq)$ cohomology with compact supports and with
rational coefficients on the category of para-compact Hausdorff spaces.
We use the notation $e_c(-)$ for the compactly supported Euler
characteristic $e_c(-) = \sum_i (-1)^i \dim\rH^i_c(-,\qq)$; this is a
well-defined integer for separated $\cc$-schemes of finite type.
\begin{lem} \label{prelthird} Let $M,N$ be separated $\cc$-schemes of
  finite type. \\
(i) If $Z$ is a closed subscheme of $M$, then $e_c(M) = e_c(Z)
+ e_c(M - Z)$. \\
(ii) If $M \to N$ is a finite \'etale cover 
of degree $n$ then $e_c(M)= n \cdot e_c(N)$.
\end{lem}
\begin{proof} The first statement follows from the long exact sequence of
compactly supported cohomology:
\[ 
\cdots \to \rH_c^i(M - Z) \to \rH_c^i(M) \to \rH_c^i(Z) \to 
\rH_c^{i+1}(M-Z) \to \cdots 
\]
As to the second statement, we may assume first of all that $M$ and
$N$ are connected.  Second, we may reduce to the case that $M \to N$
is Galois. Indeed, let $P \to N$ be a Galois closure of $M \to N$, and
denote by $G$ the group of automorphisms of $P$ such that $N =
P/G$. Let $H$ be the subgroup of $G$ such that $M=P/H$. If the result
is true for Galois covers, we find:
\[ 
e_c(N) = \frac{1}{\# G} e_c(P) = \frac{\# H}{\# G} e_c(M) = \frac{1}{n}
e_c(M) 
\]
and the result also follows in the general case. So let's assume that
$M \to N$ is Galois, with group $G$. If $V$ is a $\qq[G]$-module of
finite type, let $[V]$ be the class of $V$ in the Grothendieck group
of such modules. More generally, if $V$ is a $\zz$-graded
$\qq[G]$-module of finite type, like $\rH_c(M)$ for example, then we
denote by $[V]$ the class of $\sum_i (-1)^i V^i$. Now remark that $G$
acts freely on $M$, hence by the Lefschetz trace formula for compactly
supported cohomology (see \cite[Theorem 3.2]{dl}) we have for all
non-trivial $g \in G$ that $\sum_i (-1)^i \trace(g, \rH_c^i(M)) =
0$. By character theory it follows that $[\rH_c(M)]$ is a multiple of
$[\qq[G]]$, the class of the regular representation of $G$, say
$[\rH_c(M)]= m \cdot [\qq[G]]$ with $m \in \zz$.  Since we also have
that $\rH_c(N)=\rH_c(M)^G $ we get $e_c(N) = \dim_\qq \rH_c(M)^G = m$.
As $e_c(M)=\dim_\qq \rH_c(M) = m \cdot \#G$ the result follows.
\end{proof}
Finally we want to work with the Grothendieck-Riemann-Roch theorem.
We recall the statement and all notions that go into it. Let $M,N$ be
smooth quasi-projective varieties over~$\cc$. One has a Grothendieck
group $K_0(M)$ for coherent sheaves on~$M$. This group is isomorphic
to its analogue for locally free $O_M$-modules of finite rank, and
therefore, it has a natural ring structure. There is also a Chow ring
$\CH(M)$, coming with a natural grading. For $p \colon M \to N$ a
projective morphism one has a map $p_! \colon K_0(M) \to K_0(N)$ given
by $p_!([\FF]) = \sum_i (-1)^i [\rR^i p_* \FF ]$. Also one has a map
$p_* \colon \CH(M) \to \CH(N)$ given by proper push-forward of
cycles. The Chern character $\ch$ gives a ring homomorphism $\ch
\colon K_0(M)_\qq \to \CH(M)_\qq$. Each coherent sheaf $\FF$ on $M$
has a Todd class $\td(\FF)$ in $\CH(M)_\qq$.  The Todd class $\td(M)$
of $M$ is by definition the Todd class of the tangent bundle $T_M$
of~$M$.

The Grothendieck-Riemann-Roch theorem reads as follows.
\begin{thm}
Let $M,N$ be smooth quasi-projective varieties over~$\cc$. Let $p
\colon M \to N$ be a projective morphism and let $\FF$ be a coherent
sheaf on~$M$. Then the equality:
\[
\ch( p_! \FF) \cdot \td(N) = p_*( \ch(\FF) \cdot \td(M) )
\]
holds in~$\CH(N)_\qq$.
\end{thm}
We recall the formulas:
\[
\ch(\FF) = c_0(\FF) + c_1(\FF) + \frac{1}{2}(c_1^2(\FF) - 2c_2(\FF)) +
\mathrm{h.o.t.}
\]
and:
\[
\td(\FF) = 1 + \frac{1}{2} c_1(\FF) + \frac{1}{12} (c_1^2(\FF) + c_2(\FF)) +
\mathrm{h.o.t.}
\]
We have $c_0(\FF)= \rank (\FF)$ if $\FF$ is locally free. Finally
$c_1(\FF) = c_1(\det \FF)$. In particular $c_1(\det \rR^\cdot f_*
O_{Y'}) = c_1(f_! O_{Y'})$.

\section{Proof of Theorem \ref{main}}\label{sec3}
We recall that we assume $V$ to be connected.
We start by deriving a useful expression for $c_1(f_! O_{Y'})$. 
We recall that the singular points of $Y$ are cyclic quotient
singularities. According to \cite[Proposition~III.3.1]{bpv} 
such singularities are
rational, i.e. we have:
\[ 
\rho_*O_{Y'} = O_Y \, , \quad \rR^i \rho_* O_{Y'} = 0 \quad \textrm{for} \,
i>0 \, . 
\]
Using the Leray spectral sequence we find, writing $\bar{\pi} = \pi
\rho$, that $\rR^i \bar{\pi}_* O_{Y'} = \rR^i \pi_* O_Y$ for all
$i$. As $\pi$ is finite we obtain:
\[ 
\bar{\pi}_* O_{Y'} = \pi_* O_Y \, , \quad \rR^i \bar{\pi}_* O_{Y'} = 0 \quad
\textrm{for} \, i>0 \, . 
\]
Applying then the Leray spectral sequence to the diagram:
\[ 
\xymatrix{ Y'   \ar[dr]_{f}    \ar[r]_{\bar{\pi}} &   X \ar[d]_{h}\\
  & C     } 
\] 
we obtain $\rR^i f_* O_{Y'} = (\rR^i h_*)(\bar{\pi}_*O_{Y'})=(\rR^i
h_*)(\pi_*O_Y)$ for all $i$ and hence:
\[ f_! O_{Y'} = h_! (\pi_* O_Y) \, . \]
The Grothendieck-Riemann-Roch theorem then gives
\[ \ch(f_! O_{Y'}) \cdot \td(C) = h_* ( \ch(\pi_*O_Y) \cdot \td(X) ) \, . \]
We recall that we write $d$ for the degree of $\pi$. 
Also we recall that the sheaf $\pi_*O_Y$ is locally free of rank $d$. 
Comparing terms in degree~$0$ therefore yields:
\[
c_0(f_! O_{Y'}) = h_*(d \cdot \td(X)_{(1)} + c_1(\pi_* O_Y) ) \, .
\]
On the other hand the Grothendieck-Riemann-Roch theorem applied directly to $f$
gives:
\[
\ch(f_! O_{Y'}) \cdot \td(C) = f_*( \ch(O_{Y'}) \cdot \td(Y')) 
= f_*(\td(Y'))\, .
\]
Comparing terms in degree~$1$ we find:
\[
c_1(f_! O_{Y'}) + c_0(f_! O_{Y'}) \cdot \td(C)_{(1)} 
= f_*( \td(Y')_{(2)}) \, .
\]
Combining with our previous expression for $c_0(f_! O_{Y'})$ we get:
\[
c_1(f_! O_{Y'}) = f_*(\td(Y')_{(2)}) - 
h_* ( d \cdot \td(X)_{(1)} + c_1(\pi_* O_Y)) \cdot \td(C)_{(1)}
\]
hence:
\[
\deg \det \rR^\cdot f_* O_{Y'} 
=\deg \left\{ \frac{1}{12} f_* \left(
c_1^2(T_{Y'}) + c_2(T_{Y'}) \right) 
- h_* ( d \cdot \td(X)_{(1)} 
+ c_1(\pi_* O_Y)) \cdot \td(C)_{(1)} \right\} \, . 
\]
We are done once we show that the degrees of $c_1^2(T_{Y'})$,
$c_2(T_{Y'})$ and of $h_*(c_1 (\pi_* O_Y)) \cdot \td(C)_{(1)}$ are
bounded from above and below by linear polynomials in $d$ with
coefficients depending only on $D$ and~$h$. We start by considering
the term involving $c_1(\pi_* O_Y)$. As before write $D = \sum_{i \in
I} D_i$ for the decomposition of $D$ into its prime components. Define
$R$ to be the Weil divisor, supported on $\pi^{-1}(D)$, given as
follows: let $D_{ij}$ be a component of $\pi^{-1}(D)$ mapping onto
$D_i$, then the multiplicity of $D_{ij}$ in $R$ is $(e_{ij}-1)$. Put
$B := \pi_* R$. Note that we have a trace pairing $\pi_*O_Y
\otimes_{O_X} \pi_*O_Y \to O_X$.  This induces a mono-morphism $(\det
\pi_*O_Y)^{\otimes 2} \rightarrowtail O_X$, identifying $(\det \pi_*
O_Y)^{\otimes 2}$ with the ideal sheaf $O_X(-B)$ of $B$, as a local
computation (see e.g. \cite[III, \S 6, Proposition 13]{ser}) shows.
We obtain $c_1(\pi_* O_Y) = -\frac{1}{2} [B]$ in $\CH(X)_\qq$ so we
are done for this term if we could show that the multiplicity of each
$D_i$ in $B$ is bounded linearly in $d$. But this multiplicity is
$\sum_{j \in J_i} (e_{ij}-1)f_{ij}$ with $f_{ij}$ the degree of
$D_{ij}$ over $D_i$ and this is bounded by $d$.

Next we consider the term $c_2(T_{Y'})$. We recall
that by a version of the Gauss-Bonnet formula (see e.g. \cite[p.~416]{gh}) we
have $\deg c_2(T_{Y'}) = e_c(Y')$, the topological Euler characteristic of
$Y'$. For each $i \in I$ write $d_i := \sum_{j \in J_i} f_{ij}$. 
By Lemma \ref{prelfirstpartone} the map $\pi^{-1}(D-D^\sing) \to D-D^\sing$ is
\'etale so we have, invoking Lemma \ref{prelthird}:
\begin{eqnarray*} e_c(Y) & = & e_c(\pi^{-1}U) +
e_c(\pi^{-1}D) \\ 
& = & d e_c(U) + \sum_{i \in I} d_i e_c(D_i - D^\sing)
+ e_c(\pi^{-1}D^\mathrm{sing}) \\
& = & d e_c(U)   + \sum_{i \in I} d_i e_c(D_i - D^\sing) +
\#  \pi^{-1}D^\mathrm{sing}     
\end{eqnarray*}
with $D^\mathrm{sing}$ the singular locus of $D$.
This shows that $e_c(Y)$ is bounded from above and below 
by linear polynomials in $d$ with
coefficients depending only on $D$. Now
$ e_c(Y') = e_c(Y) + s $,
where $s$ is the total number of exceptional components $E_1,\ldots,E_s$ 
of $Y' \to Y$. 
If $y$ is a singular point of $Y$ of type $A_{n,q}$ say and mapping to $x$
on $X$ then by Lemma
\ref{prelsecond} the number of exceptional components above $y$ is bounded
from above by $n$. By Lemma \ref{prelfirstparttwo} this is again 
bounded from
above by the local degree $d_y$ of $y$ over $x$. Since for any $x$ on $X$ we
have $\sum_{y \colon y \mapsto x} d_y = d$ we obtain that $e_c(Y')$ is 
bounded from above and below 
by linear polynomials in $d$ with coefficients depending only on $D$.

Finally we consider $c_1^2(T_{Y'})$. Note that we can write $\deg
c_1^2(T_{Y'}) = (K_{Y'},K_{Y'})$, the self-intersection number of the divisor
$K_{Y'}$ on $Y'$. We compute this self-intersection number. 
By \cite[Theorem~4.1]{vis} 
the normal surface $Y$ is an Alexander scheme, implying (cf. op. cit., Note~2.4)
among other things that for the proper maps $\rho \colon Y' \to Y$ and $\pi \colon Y \to X$ one has
a projection formula for Weil divisors, provided that one works on $Y$ with the
intersection theory with $\qq$-coefficients as in \cite[Section~IIb]{mu}. Thus
we compute
\begin{eqnarray*} (K_{Y'},K_{Y'}) & = & (\rho^*K_Y + \sum_i a_i E_i, K_{Y'}) \\
	& = & (\rho^*K_Y, K_{Y'}) + \sum_i a_i(b_i-2) \\
	& = & (K_Y,K_Y) + \sum_i a_i(b_i-2) \, . 
\end{eqnarray*}
But $K_Y = \pi^*K_X + R$ so
\begin{eqnarray*} 
(K_Y,K_Y) & = & d \cdot (K_X,K_X) + 2 (\pi^*K_X,R) + (R,R) \\
	& = & d \cdot (K_X,K_X) + 2 (K_X,B) + (R,R) \, . 
\end{eqnarray*}
We are done once we show that $\sum_i a_i(b_i-2)$, $(K_X,B)$ and $(R,R)$ are
bounded from above and below by linear polynomials in $d$ with coefficients
depending only on $D$. We start with the term 
$\sum_i a_i(b_i-2)$. By Lemma \ref{prelsecond} the contribution coming from
one singularity $y$ is bounded from above by~$2$ and from below by $-n$
with $n$ determined as usual by the type of $y$. Again, since $n$ is bounded
by $d_y$ and $\sum_{y \colon y \mapsto x} d_y =d$ for all $x$ on $X$ we get
in total that the sum $\sum_i a_i(b_i-2) $ 
is bounded by at most linear polynomials in $d$ with coefficients depending
only on $D$.

The intersection number 
$(K_X,B)$ is bounded linearly in~$d$ by our description of $B$
given earlier in this proof. 

As for $(R,R)$, we obtain from 
Lemmas \ref{prelfirstpartone} and \ref{prelfirstparttwo} 
that the irreducible components of the inverse
image under $\pi$ of a $D_i$ are disjoint and hence we can write:
\[ (R,R)=\sum_{i,j} (e_{ij}-1)^2(D_{ij},D_{ij}) + 
\sum_{\substack{(i,i'),j,j'\\i \neq i'}} 
(e_{ij}-1)(e_{i'j'}-1)(D_{ij},D_{i'j'}) \, .
\]
Now we have, for each $i \in I$, that $\pi^*D_i = \sum_j e_{ij} D_{ij}$ so on
the one hand for a given $j_0$:
\[ (D_{ij_0},\pi^*D_i) = \sum_j e_{ij} (D_{ij},D_{ij_0}) =
e_{ij_0}(D_{ij_0},D_{ij_0}) \]
by the disjointness of the $D_{ij}$ and on the other hand:
\[ (D_{ij_0},\pi^*D_i) = (\pi_* D_{ij_0},D_i) = f_{ij_0}(D_i,D_i) \]
by the projection formula. Thus:
\[ (D_{ij_0},D_{ij_0}) = \frac{f_{ij_0}}{e_{ij_0}} (D_i,D_i) \]
and hence for a given $i$:
\[ \sum_j (e_{ij}-1)^2(D_{ij},D_{ij}) = \sum_j (e_{ij}-1)^2
\frac{f_{ij}}{e_{ij}}(D_i,D_i) \, . \]
Remark that $0 \leq \sum_j (e_{ij}-1)^2
\frac{f_{ij}}{e_{ij}} < d$ and we are done for the first term 
$\sum_{i,j} (e_{ij}-1)^2(D_{ij},D_{ij})$. Finally we can write:
\[
\sum_{\substack{(i,i'),j,j'\\ i \neq i'}} (e_{ij}-1)(e_{i'j'}-1)
(D_{ij},D_{i'j'}) = \sum_{\substack{(i,i')\\ i \neq i'}}
\sum_{x \in D_i \cap D_{i'}} \sum_{y \mapsto x} \sum_{j,j'}
(e_{ij}-1)(e_{i'j'}-1)(D_{ij},D_{i'j'})_y \, .
\]
But as is stated in Lemma \ref{prelfirstparttwo} for each $ y \mapsto
x$ with $x \in D_i \cap D_{i'}$ there is exactly one pair $(j,j')$
such that $(D_{ij},D_{i'j'})_y \neq 0$. So, the summation over $j$ and
$j'$ can be replaced by a single term with indices $j(y)$
and~$j'(y)$. We can compute $(D_{ij(y)},D_{i'j'(y)})_y$ as follows.
Let $A_{n,q}$ be the type of~$y$. By what we said before
Lemma~\ref{prelsecond}, $n{\cdot}D_{ij(y)}$ is a Cartier divisor
at~$y$. But then as $D_{i'j'(y)}$ is smooth the intersection number
$(D_{ij(y)},D_{i'j'(y)})_y$ is just given as $1/n$ times the valuation in
$O_{D_{i'j'(y)},y}$ of a local function defining $n \cdot D_{ij(y)}$ around
$y$ on $Y$. Since this function is a local coordinate around $y$ on
$D_{i'j'(y)}$ we find $(D_{ij(y)},D_{i'j'(y)})_y$ to be equal to~$1/n$. By
Lemma~\ref{prelfirstparttwo} there exist positive integers $m_1,m_2$
such that $e_{ij(y)}=nm_1$, $e_{i'j'(y)}=nm_2$, $d_y = nm_1m_2$ hence
\begin{eqnarray*} (e_{ij(y)}-1)(e_{i'j'(y)}-1) (D_{ij(y)},D_{i'j'(y)})_y 
& \leq &
e_{ij(y)} e_{i'j'(y)} (D_{ij(y)},D_{i'j'(y)})_y \\ 
& = & (nm_1)(nm_2)/n = nm_1m_2 =
d_y \, . 
\end{eqnarray*}
Keeping in mind that $\sum_{y \colon y \mapsto x} d_y = d$ for all $x$ on
$X$ we find that:
\[ \sum_{\substack{(i,i')\\ i \neq i'}}
\sum_{x \in D_i \cap D_{i'}} \sum_{y \mapsto x} 
(e_{ij(y)}-1)(e_{i'j'(y)}-1) (D_{ij(y)},D_{i'j'(y)})_y \]
is bounded by a linear polynomial in $d$ with coefficients depending only on
$D$. This finishes the proof.

\section{Alternative proofs} \label{sec4}
H{\'e}l{\`e}ne Esnault and Eckart Viehweg have proposed another proof
of (the upper bound of) Theorem~\ref{main}, based on Arakelov's
inequality. In fact, their method leads to a more precise version.
\begin{thm}[Esnault-Viehweg]\label{thm_esnault-viehweg}
Assumptions as in Theorem~\ref{main}. Assume moreover that $h\colon
X\to C$ is semi-stable (i.e., the singularities in its fibres are
ordinary double points), with connected fibres, and that $D=D^\hor
+h^{-1}D_C$ with $D^\hor\to C$ {\'e}tale and with $D_C$ a divisor
on~$C$. Let $g(C)$ be the genus of~$C$, and $g(F)$ the genus of any
smooth fibre~$F$ of~$h\colon X\to C$. Let $S\subset C$ be the set of
$s\in C$ such that the fibre $X_s$ of $h$ is singular. Then, for every
finite {\'e}tale $\pi\colon V\to U$ we have:
\[
\deg\det\rR^\cdot f_* O_{Y'} \leq 
\left(g(F)+ \frac{1}{2}(D^\hor,F)\right){\cdot}
\left(g(C)+2\# D_C + \frac{1}{2}(1+\# S)\right){\cdot}\deg(\pi)\, .
\]\
\end{thm}
\begin{proof}
We may and do assume that $V$ is connected. 

The results of the beginning of Section~\ref{sec2} show the following
statements. The reduced fibres of $f\colon Y'\to C$ are normal
crossings divisors on~$Y'$. For $s\in C$ with $s\not\in D_C$, the
fibre $f^{-1}s$ is semi-stable. Let $\Sing(f)$ denote the singular set
of~$f$, i.e., the set where the tangent map of $f$ vanishes. Then
$f\Sing(f)\subset S\cup D_C$. 

We write $\omega_{Y'}:=\Omega^2_{Y'}$ and $\omega_C:=\Omega^1_C$ for
the dualising sheaves of $Y'$ and~$C$. We define
$\omega_{Y'/C}:=\omega_{Y'}\otimes(f^*\omega_C)^\vee$, the relative
dualising sheaf for~$f$. We note that $\omega_{Y'/C}$ coincides with
$\Omega^1_{Y'/C}$ on the complement of~$\Sing(f)$. By
\cite[Theorem~5.1]{lipman}, we have $(\rR^1f_*O_{Y'})^\vee =
f_*\omega_{Y'/C}$. On the other hand, $f_*O_{Y'}$ is the $O_C$-algebra
corresponding to the Stein factorisation $Y'\to \wt{C}\to C$
of~$f$. As $Y'$ is reduced, the same is true for~$\wt{C}$, hence the
trace form gives an injection of $(\det O_{\wt{C}})^{\otimes 2}$ into
$O_C$, hence $\deg\det f_*O_{Y'}\leq 0$.  We get:
\begin{eqn}\label{eqn4.2}
\deg\det\rR^{\cdot}f_*O_{Y'} = \deg c_1 \rR^0f_*O_{Y'} 
-\deg c_1\rR^1f_*O_{Y'} \leq -\deg c_1\rR^1f_*O_{Y'}
\end{eqn}
The coherent $O_C$-module $\rR^1f_*O_{Y'}$ sits in an exact sequence:
\begin{eqn}\label{eqn4.3}
0 \to (\rR^1f_*O_{Y'})_\tors \to \rR^1f_*O_{Y'} \to
\ol{\rR^1f_*O_{Y'}}
\to 0\, ,
\end{eqn}
given by its torsion submodule and its locally free quotient. This
implies:
\begin{eqn}\label{eqn4.4}
-\deg c_1(\rR^1f_*O_{Y'}) = -\deg c_1(\ol{\rR^1f_*O_{Y'}}) 
-\deg c_1(\rR^1f_*O_{Y'})_\tors \leq -\deg c_1(\ol{\rR^1f_*O_{Y'}})\, .
\end{eqn}
As $(\rR^1f_*O_{Y'})^\vee=(\ol{\rR^1f_*O_{Y'}})^\vee$, (\ref{eqn4.2})
and~(\ref{eqn4.4}) give:
\begin{eqn}\label{eqn4.5}
\deg\det\rR^{\cdot}f_*O_{Y'} \leq \deg\det f_*\omega_{Y'/C}\, .
\end{eqn}

The main idea of the proof is now to use Arakelov's inequality (to be
explained below), but for that we need a finite base change $b\colon
C'\to C$, with $C'$ a smooth connected projective complex curve, after
which $Y'$ admits a semi-stable model $Y''\to C'$. This precisely means
that the ramification indices of $b$ at the $s\in D_C$ must be
sufficiently divisible. We pick any $s_0$ in $C-D_C$. Then
$\pi_1(C-(\{s_0\}\cup D_C))$ is freely generated by the usual kind of
generators $a_i, b_i$ with $1\leq i\leq g(C)$, and $\gamma_s$ for
$s\in D_C$. This shows that a $b\colon C'\to C$ as required does
exist, unramified outside $E_C:=\{s_0\}\cup D_C$. We pick such a~$b$.

Let $f'\colon Y''\to C'$ be obtained by pull-back via $b$ of $f\colon
Y'\to C$, then normalisation, and then minimal resolution of
singularities. Then $f'$ is semi-stable (see~\cite{ast86},
or \cite[p.~83]{esnault-viehweg}), 
and we have a commutative diagram:
\begin{eqn}\label{eqn4.6}
\xymatrix{
Y'' \ar[r]^{b'}\ar[d]^{f'} & Y'\ar[d]^f\\
C'\ar[r]^b & C
}  
\end{eqn}
\begin{lem}\label{lem4.7}
In this situation, there is an injection of $b^*f_*\omega_{Y'/C}$ into
$b^*O_C(E_C)\otimes f'_*\omega_{Y''/C'}$.
\end{lem}
\begin{proof}
To start with, the projection formula gives $f_*\omega_{Y'/C} =
(f_*\omega_{Y'})\otimes\omega_C^\vee$, and, similarly,
$f'_*\omega_{Y''/C'} =
(f'_*\omega_{Y''})\otimes\omega_{C'}^\vee$. Pull-back of 2-forms along
$b'$ gives a morphism of $O_{Y'}$-modules
$(b')^*\omega_{Y'}\to\omega_{Y''}$, which is generically an
isomorphism. Applying $f'_*$ to this morphism, and composing with the
natural morphism $b^*f_*\omega_{Y'}\to f'_*(b')^*\omega_{Y'}$ gives a
morphism of locally free $O_{C'}$-modules $b^*f_*\omega_{Y'}\to
f'_*\omega_{Y''}$ which is generically an isomorphism. Pull-back of
1-forms via $b$ gives a morphism of invertible $O_{C'}$-modules
$b^*\omega_C\to \omega_{C'}$, which is an isomorphism
outside~$b^{-1}E_C$. If $P$ is in $C'$ and $\omega$ is a generating
1-form at $b(P)$, then $b^*\omega$ has a zero of order $e(P)-1$ at
$P$, where $e(P)$ is the ramification index of~$b$ at~$P$. It follows
that we have an inclusion of $b^*O_C(-E_C)\otimes\omega_{C'}$ into
$b^*\omega_{C}$. Dually, this gives an inclusion of
$b^*\omega_{C}^\vee$ into
$b^*O_C(E_C)\otimes\omega_{C'}^\vee$. Combining all this gives:
\[
b^*f_*\omega_{Y'/C} = b^*f_*\omega_{Y'}\otimes b^*\omega_C^\vee\to
f'_*\omega_{Y''}\otimes b^*O_C(E_C)\otimes\omega_{C'}^\vee = 
b^*O_C(E_C)\otimes f'_*\omega_{Y''/C'}\, .
\]
\end{proof}
We are now ready to invoke Arakelov's inequality
(see~\cite[p.~58]{ast86}):
\begin{eqn}\label{eqn4.8}
\deg\det(f'_*\omega_{Y''/C'}) \leq
\frac{1}{2}\rank(f'_*\omega_{Y''/C'}){\cdot}
\left(2g(C')-2+\# f'\Sing f'\right)\, .
\end{eqn}
The rank of $f'_*\omega_{Y''/C'}$ equals that of $f_*\omega_{Y'/C}$,
both being the sum of the genera of the connected components of
the geometric generic fibre of~$f$.
Combining this with our previous inequalities, we obtain:
\begin{eqn}\label{eqn4.9}
\begin{aligned}
\deg\det\rR^{\cdot}f_*O_{Y'} & \leq \deg\det f_*\omega_{Y'/C}
= \frac{1}{\deg b}\deg\det(b^*f_*\omega_{Y'/C}) \leq \\
& \leq \frac{1}{\deg b}
\deg\det\left(b^*O_C(E_C)\otimes f'_*\omega_{Y''/C'}\right) = \\
& = \frac{1}{\deg b} \left( 
\deg(b^*O_C(E_C)){\cdot}\rank f'_*\omega_{Y''/C'} + 
\deg\det f'_*\omega_{Y''/C'}\right) \leq \\
& \leq (\# E_C)\rank f_*\omega_{Y'/C} + 
\frac{1}{2\deg b}(\rank f_*\omega_{Y'/C}){\cdot}(2g(C')-2+\#(f'\Sing
f'))\, .
\end{aligned}
\end{eqn}
Finally, we bound the quantities in the last term. Letting $d_j$ be
the degrees of the connected components $Z_j$ of the geometric generic
fibre of~$f$ over the geometric generic fibre of~$h$, and $g_j$ the
genera of the~$Z_j$, we have $\sum_i d_i = \deg(\pi)$, and Hurwitz's
formula gives:
\[
2g_i-2 = d_i{\cdot}(2g(F)-2)+\deg R_i, \quad \deg R_i \leq
(d_i-1)(D^\hor,F)\, .
\]
This leads to:
\begin{eqn}\label{eqn4.10}
\rank f_*\omega_{Y'/C} \leq (g(F)+(D^\hor,F)/2)\deg(\pi) \, .
\end{eqn}
For $g(C')$, we note that $b\colon C'\to C$ is unramified
outside~$E_C=\{s_0\}\cup D_C$. Hurwitz's formula gives:
\begin{eqn}\label{eqn4.11}
\frac{1}{\deg b}(2g(C')-2) \leq 2g(C)-2 + 1+\# D_C \, .
\end{eqn}
At the beginning of the proof we noticed that $f\Sing f$ is contained
in $S\cup D_C$. Therefore, $f'\Sing f'$ is contained in $b^{-1}(S\cup
D_C)$. So:
\begin{eqn}\label{eqn4.12}
\frac{1}{\deg b}\#(f'\Sing f') \leq \# S + \# D_C\, .
\end{eqn}
Combining (\ref{eqn4.9})--(\ref{eqn4.12}) we get:
\begin{eqn}\label{eqn4.13}
\deg\det \rR^{\cdot}f_*O_{Y'} \leq 
\left(g(F)+ \frac{1}{2}(D^\hor,F)\right){\cdot}
\left(g(C)+2\# D_C + \frac{1}{2}(1+\# S)\right){\cdot}\deg(\pi)\, .
\end{eqn}
This ends our proof of Theorem~\ref{thm_esnault-viehweg}.
\end{proof}
We remark that instead of invoking Arakelov's inequality it is also
possible, again at least for the upper bound implied by Theorem
\ref{main}, to invoke the Bogomolov-Miyaoka-Yau inequality. 
Indeed, our work done 
at the beginning of Section \ref{sec3} showed that Theorem \ref{main} 
can be reduced to providing linear bounds in $\deg(\pi)$ for the degrees of
each of the three terms $c_1^2(T_{Y'})$,
$c_2(T_{Y'})$ and $h_*(c_1 (\pi_* O_Y)) \cdot \td(C)_{(1)}$. The
latter two were relatively easily to deal with, whereas the term 
$c_1^2(T_{Y'})$ required significantly more work. Now instead of
calculating the $c_1^2$ directly, one can also remark that there 
exist \emph{a priori} inequalities relating the $c_2$ and the $c_1^2$.
Table~10 in Chapter VI of \cite{bpv} first of all shows that for
smooth compact connected complex surfaces which are not of general
type $c_1^2$ is bounded absolutely from above by~9. The
Bogomolov-Miyaoka-Yau inequality \cite[Theorem VII.4.1]{bpv} says that
for a smooth compact connected complex surface which is of general type
the inequality $c_1^2 \leq 3c_2$ holds. Invoking these results one
obtains yet another proof of Theorem \ref{main}.

\section{An arithmetic analogue} \label{arithmetic}
In \cite{report} an algorithm is given that computes the
$\GL_2(\ff_\lambda)$ Galois representations associated to a given
normalised Hecke eigenform $f$ of level $1$, in time polynomial in
$\#(\ff_\lambda)$, if one admits the Generalised Riemann Hypothesis
for number fields. Here $\lambda$ runs through the finite degree~$1$
places of the field of coefficients of the form.  By a famous argument
due to R. Schoof, this leads to an algorithm that on input a prime
number $p$ computes the $p$-th coefficient of the Fourier development
of $f$, in time polynomial in $\log p$. As a consequence, the number
of vectors with half length-squared equal to $p$ in a fixed even
unimodular lattice can be computed in time polynomial in~$\log p$.

Generalisations of the above results seem possible in various
different directions. For example, one could look at the case of
mod-$\ell$ Galois representations occurring in the \'etale cohomology
of a given smooth, projective and geometrically connected surface $S$
over $\qq$. Letting $\ell$ be a prime number, one has the cohomology
groups $ \rH^i(S_{\Qbar, \mathrm{et}}, \ff_\ell) $ for $ 0 \leq i \leq
4$, being finite dimensional $\ff_\ell$-vector spaces with
$\mathrm{Gal}(\Qbar/\qq)$-action. It seems reasonable to suspect that,
again, there is an algorithm that on input a prime $\ell$ computes
these cohomology groups, with their $\mathrm{Gal}(\Qbar/\qq)$-action,
in time polynomial in $\ell$.  Once such an algorithm is known, one
also has an algorithm that, on input a prime $p$ of good reduction of
$S$, gives the number of points $\# S(\ff_p)$ of $S$ over $\ff_p$ in
time polynomial in $\log p$. This result would be of interest because
the known $p$-adic algorithms for finding such numbers have running
time exponential in $\log p$.

The idea in \cite{report} to compute mod-$\ell$ \'etale cohomology is
to trivialise the sheaves involved, using suitable covers of (modular)
curves, and to reduce to computing in the $\ell$-torsion of their
Jacobians. In our case, using a Lefschetz fibration, one can first
reduce to computing cohomology groups $ \rH^1( U_{\Qbar}, \FF_l) $
where $U$ is a non-empty open subscheme of $\pp^1_\qq$ determined by $S$
and the chosen Lefschetz fibration and where $\FF_l$ are certain
\'etale locally constant sheaves of $\ff_\ell$-vector spaces of fixed
dimension, say $r$. For each $\ell$ let $V_\ell :=
\underline{\mathrm{Isom}}_U(\ff_\ell^r, \FF_l)$. Then each cover
$V_\ell \to U$ is finite Galois with group $G \cong \GL_r(\ff_\ell)$,
and the group $\rH^1( U_{\Qbar}, \FF_l) $ can be related to
$\rH^1(V_{\ell,\Qbar},\ff_\ell^r)$ which sits in the $\ell$-torsion of
the Jacobian of the smooth projective model $\overline{V_\ell}$ of 
$V_\ell$. It is our hope that methods as in \cite{report} can show
that we have a polynomial algorithm for computing these cohomology
groups once we have a bound for the Faltings height of
$\overline{V_\ell}$ that is polynomial in~$\ell$. 

A recent result of 
Bilu and Strambi \cite[Theorem 1.2]{bs} suggests at least 
the existence of an 
exponential bound: it implies that if  $Y_\Qbar \to
\pp^1_\Qbar$ is any connected degree $d \geq 2$ cover, 
unramified outside a finite subset $B$ of $\pp^1_\Qbar$, 
then $Y_\Qbar$ has
a plane model given by an equation $F(u,v)$ with coefficients in
$\Qbar$ such that
$\deg_u F = d$, $\deg_v F \leq d \cdot (\# B/2)$, 
and such that the affine logarithmic height of $F$ is bounded from
above by:
\[ (h+1)(d^3 \cdot \# B)^{5d^2 \cdot \#B + 12d} \, . \]
Here $h$ is the maximum of the
logarithmic projective heights of the elements of $B$.

Note, however, that in our situation we can take more restrictive
assumptions than in the result of Bilu and Strambi. For one thing, our
covers $V_\ell \to U$ are defined over $\qq$. For another, our covers
$V_\ell \to U$ 
have the property that there is a nonempty 
open subscheme $U'$ of $\pp^1_\zz$ containing $U$ 
such that each $V_\ell$ extends to a
finite \'etale cover of $U'_{\zz[1/\ell]}$. We would like therefore to
propose the following arithmetic analogue of the main theorem of this
note.
\begin{conj}\label{conj}
Let $U\subset\pp^1_\zz$ be a nonempty open subscheme. Then there are
integers~$a$ and~$b$ with the following property. For any prime
number~$\ell$, and for any connected finite {\'e}tale cover $\pi\colon
V\to U_{\zz[1/\ell]}$, the absolute value of the Faltings height of the
normalisation of $\pp^1_\qq$ in the function field of $V$ is bounded
by $\deg(\pi)^a{\cdot}\ell^b$. 
\end{conj}
As an example, let us mention that for the family of modular curves
$X_1(\ell)$, all covers of the $j$-line, where one can take
$U=\pp^1_\zz-\{0,1728,\infty\}$, it is proved in~\cite{report} that
the Faltings height is bounded above as $O(\ell^2\log \ell)$. It is
tempting to interpret $\ell^2$ (up to a constant factor) as the degree
of $X_1(\ell)$ over the $j$-line, and the factor $\log \ell$ as coming
from the ramification at~$\ell$. As in our application the degree of
$\pi$ itself depends polynomially on~$\ell$, it is irrelevant for this
application if the bound in the conjecture is
$\deg(\pi)^a{\cdot}\ell^b$ or $\deg(\pi)^a{\cdot}(\log \ell)^b$. Of
course, one may also ask about the conjecture with this stronger
bound.

\noindent Email address Bas Edixhoven: \verb+edix@math.leidenuniv.nl+ \\
\noindent Email address Robin de Jong: \verb+rdejong@math.leidenuniv.nl+ \\
\noindent Email address Jan Schepers: \verb+janschepers1@gmail.com+ \\  

\noindent Mathematisch Instituut \\
Universiteit Leiden \\
Postbus 9512 \\
2300 RA Leiden \\
Nederland

\vspace{\baselineskip}
\noindent Departement Wiskunde\\ 
Katholieke Universiteit Leuven\\ 
Celestijnenlaan 200b\\
3001 Heverlee\\
Belgi\"e

\end{document}